\documentclass[reqno,10pt,a4paper]{amsart}
\usepackage{amssymb,amsmath, amsfonts}
\usepackage{amscd}
\usepackage{latexsym}

\usepackage[cmtip,matrix,arrow]{xy}
\UseComputerModernTips
\CompileMatrices

\DeclareMathOperator{\Hom}{Hom}
\DeclareMathOperator{\Ext}{Ext}

\DeclareMathOperator{\dHom}{\dim Hom}
\DeclareMathOperator{\dExt}{\dim Ext}

\newcommand{\N}{\mathbb{N}}

\newcommand{\frob}{^{\mathrm F}}

\DeclareMathOperator{\Ind}{Ind}

\newcommand{\GL}{\mathrm{GL}}
\newcommand{\SL}{\mathrm{SL}}

\newcommand{\D}{\Delta}
\newcommand{\ve}{\varepsilon}
\newcommand{\cM}{\mathcal{M}}
\newcommand{\cN}{\mathcal{N}}

\newcommand{\cL}{\mathcal{L}}
\newcommand{\alphac}{\alpha\check{\ }\,}

\begin{document}
\theoremstyle{plain}
\numberwithin{subsection}{section}
\numberwithin{equation}{subsection}
\newtheorem{thm}{Theorem}[subsection]
\newtheorem{prop}[thm]{Proposition}
\newtheorem{cor}[thm]{Corollary}
\newtheorem{clm}[thm]{Claim}
\newtheorem{lem}[thm]{Lemma}
\newtheorem{conj}[thm]{Conjecture}
\theoremstyle{definition}
\newtheorem{defn}[thm]{Definition}
\newtheorem{rem}[thm]{Remark}
\theoremstyle{definition}
\newtheorem{eg}[thm]{Example}

\title{Dimensions of higher extensions for {$\SL_2$}}
\author{Karin Erdmann}
\address{Mathematical Institute, 24-29 St. Giles',  Oxford OX1 3LB, UK}
\email{erdmann@maths.ox.ac.uk}
\author{Keith C. Hannabuss}
\address{Balliol College,   Oxford OX1 3BJ, UK}
\email{kch@balliol.ox.ac.uk}
\author{Alison E. Parker}
\address{School of Mathematics, University of Leeds, Leeds
LE2 9JT, UK}
\email{a.e.parker@leeds.ac.uk}



\begin{abstract}
We analyse the recursive formula found for various Ext groups for
$\SL_2(k)$, $k$ a field of characteristic $p$, and
derive various generating functions for these groups. We use
this to show that the growth rate for the cohomology of $\SL_2(k)$ is
at least
exponential. In particular, $\max \{ \dim \Ext^i_{\SL_2(k)}(k,
\Delta(a))\mid a,i \in \N \}$ has (at least) exponential growth for all $p$.
We also show that $\max \{ \dim \Ext^i_{\SL_2(k)}(k, 
\Delta(a))\mid a\in \N \}$ for a fixed $i$ is bounded.
\end{abstract}

\maketitle

\section*{Introduction}

A very general open problem in the characteristic $p$ representation
theory of algebraic groups is to determine the higher extension groups
 $\Ext_G^q(M, N)$ where $M, N$ are Weyl modules,
or simple modules (or more generally if possible).
In \cite{parker3} the third author found recursive formula for many
different $\Ext$ groups for modules in $\SL_2(k)$, $k$ an
algebraically closed field of characteristic $p$.
But no closed formula were found. More recently there has been
interest in finding upper bounds for the dimensions of $\Ext$ groups
(see for example \cite{PSboundext}). Work of  \cite{stewart} applied
the results of
\cite{parker3} to show that the growth rates of $\Ext^i_{\SL_2(k)}(k,L)$,
taken over all $L$ a simple module and $i \in \N$ is at least exponential.  
Several years ago, just after \cite{parker3} had been done, we had found
some generating functions for the dimensions of $\Ext$ groups for Weyl modules,
to investigate how large these could be, (mentioned in
\cite[section 6]{coxpar}). The recent work of
Parshall and Scott in \cite{PSboundext} and
\cite{PSisrael}, and of Stewart~\cite{stewart},
has encouraged us
to 
polish this work up to explore further questions raised by these
authors. In particular, using our generating functions, we have got an
analogue of the exponential growth found by Stewart \cite{stewart} for
all primes, but using Weyl modules rather than simples. We also show
that when we fix $i$ 
that $\max \{ \dim \Ext^i_{\SL_2(k)}(k, 
\Delta(a))\mid a\in \N \}$ is bounded (see
section \ref{sect:boundext}).

For prime $2$ we have an explicit formula for
the dimension of $\Ext^n_{\SL_2(k)}(\Delta(0), \Delta(a))$.
When $a$ is odd this is zero by block considerations,
 so let $a=2d$. We show that this is equal to 
the number of  partitions $(b_0, b_1, \ldots, b_n)$      
such that $\sum_i 2^{b_i}= d+1$ (see corollary \ref{cor:p2dimext}).
These are compatible with Stewart's results.

For $p>2$ the situation is more complicated. 
Our generating function $G(s)$ 
can be written as 
$\sum_{n\geq 0} z^nh_n(s)$ with $h_n(s)$ a power series, 
and the coefficient of $s^d$ in 
$h_n(s)$ is the dimension of 
 $\Ext^n_{\SL_2(k)}(\Delta(0), \Delta(2d))$. 
We have a recursion for $sh_n(s)$, given in 
\ref{subsect:shlabels} 
and an algorithm which is described in \ref{subsect:genn}.
Exponential growth is established for $p=2$ in Propositon
\ref{prop:dim} and  for odd primes in Lemma
\ref{lem:expgrowth} together with Lemma \ref{lem:growthbinom}. 
That is, for a fixed prime, if we let $d$ vary with $n$, then the dimensions 
of $\Ext^n_{\SL_2(k)}(\Delta(0), \Delta(2d))$ grow exponentially.

\section{Preliminaries}

\subsection{Notation}
We first briefly review some of the notation and definitions that we
will use in this paper. 
The reader is referred to \cite{jantz2} for further information. 
We let $G=\SL_2(k)$ where $k$ is an algebraically closed field of
characteristic $p$,
 and $\mathrm{F}:G \to G$ the corresponding
Frobenius morphism. We may ``twist'' $G$-modules via this morphism.
We let $X^+$ be the set of dominant weights which may be identified
with $\N$, the non-negative integers. 

For $\lambda \in X^+$, let $k_\lambda$ be the one-dimensional module
for $B$ a suitable Borel, which has weight $\lambda$. We define
$\nabla(\lambda)= \Ind_B^G(k_\lambda)$.
This module has character given by Weyl's character formula and has
simple socle $L(\lambda)$, the irreducible $G$-module of highest weight
$\lambda$. In the case of $\SL_2$ all simples are known via
Steinberg's tensor product theorem.
If we let $E$ be the 2-dimensional natural module for $\SL_2(k)$, then
$\nabla(\lambda) = S^\lambda E$, the $\lambda$th symmetric power of
$E$.
We will also use Weyl modules $\D(\lambda)$ which for our purposes can be either
thought as divided powers, so $\D(\lambda) = D^\lambda E$ or as duals
of induced modules: $\D(\lambda) = \nabla(\lambda)^*$, where $^*$ is
the usual $k$-linear dual.

The category of rational $G$-modules has enough injectives and so we may
define $\Ext^*(-,-)$ as usual by using injective resolutions.

\subsection{ Background }  
Past work of the third author \cite{parker1, parker2},  was
concerned with finding explicit bounds on the global dimension of the
Schur algebra associated to polynomial modules for $\GL_n(k)$. 
Of course it was known that such a bound should exist as the category
of $G$-modules with bounded highest weight is an example of what is
now known as a high weight category (or equivalently, is the module
category of a quasi-hereditary algebra) and as such has finite global
dimension.
The work of \cite{parker2} showed that for any algebraic group that
$$
\Ext^m_G(L(w\cdot \lambda), L(v\cdot \lambda)) = 
\Ext^m_G(\nabla(w\cdot \lambda), \Delta(v\cdot \lambda)) = 
\begin{cases}
0 & \mbox{if $m > l(w) + l(v)$,}\\
k & \mbox{if $m = l(w) + l(v)$}
\end{cases}
$$
where $\lambda$ is in in the interior of the fundamental alcove, 
$w \in W_p$, the affine Weyl group acting via the dot action on
the dominant weights and $l:W_p \to \N$ is the usual length
function on (the Coxeter group) $W_p$. 
We also have:
$$
\Ext^m_G(\D(w\cdot \lambda), \D(v\cdot \lambda)) = 
\begin{cases}
0 & \mbox{if $m > l(v) - l(w)$,}\\
k & \mbox{if $m = l(v) - l(w)$}
\end{cases}
$$
with the same notation as above.

For $\SL_2(k)$ this was proved more directly in \cite{parker1} and
can be summarised as follows:
$$
\Ext^m_{\SL_2(k)}(L(pa+i), L(pb + j)) = 
\begin{cases}
0 & \mbox{if $m > a +b $,}\\
k & \mbox{if $m = a+b$ and $pa+i$ and $pb+j$ are in the same
  $W_p$-orbit}
\end{cases}
$$
where $a$, $b$, $i$, $j \in \N$ and $ 0 \le i, j \le p-2$.
We also have:
$$
\Ext^m_{\SL_2(k)}(\D(pa+i), \D(pb + j)) = 
\begin{cases}
0 & \mbox{if $m > b -a $}\\
k & \mbox{if $m = b-a$ and $pa+i$ and $pb+j$ are in the same
  $W_p$-orbit}.
\end{cases}
$$

NB: The condition for $pa+i$ and $pb+j$ to be in the same $W_p$ orbit
is that either $a-b$ is even and $i=j$ or $a-b$ is odd and $i =
p-2-j$.
The case that will be of most use in what follows is that when $pa+i
=0$. Then $pb+j$ is in the same $W_p$-orbit (=$G$-block) as $0$ when
$b$ is even and $j=0$ or $b$ is odd and $j=p-2$.

Thus we know that most of the $\Ext$ groups are zero and 
we also know what some of the lower dimensional groups between Weyl
modules are for $\SL_2(k)$ thanks to
work of the first author  \cite{erd1}, and of Cox and the first author
\cite{coxerd}. We won't cite these results directly, but note that 
\begin{align*}
\dHom_{\SL_2(k)}(\D(pa+i), \D(pb+j)) &\le 1,\\
\dExt^1_{\SL_2(k)}(\D(pa+i), \D(pb+j)) &\le 1\\
\intertext{and}
\dExt^2_{\SL_2(k)}(\D(pa+i), \D(pb+j)) &\le 2.\\
\end{align*}
This of course rather does beg the question, can we calculate the
other $\Ext$ groups and can we find bounds on their dimensions?

The work of \cite{parker3} allows us in theory to calculate these
$\Ext$ groups but with recusive formula. While this is easy to program
into a computer, this formulation has not proved useful so far for
more theoretical results.
This paper is our attempt to put the
recursive formula of \cite{parker3} into a form which will allow us
(or others) to answer such questions as is
$\dExt^m_{\SL_2(k)}(\D(pa+i),\D(pb+j))$ bounded? If it isn't, what's
the growth rate like? The rest of the paper is devoted to applying the
theory of generating functions  to the recursions in order to give
partial answers to some of
these  questions.

\section{Some recursions} \label{sect:somerec}
Henceforth all $\Ext$ groups will be over $\SL_2(k)$ so we will drop
the subscript. (The algebraic group in any case should be self-evident from the
highest weights of the modules involved.)

We first note that for all $p$,
\begin{align}\label{eq:hom}
\Ext^m(\D(0), \D(0)) &= 0 \mbox{ for $m\geq 1$ }\\
\Hom(\D(0), \D(2b)) &= 
\begin{cases}
k & \mbox{if $b=p^{r}-1$, $r \in \N$,} \\
0 & \mbox{ otherwise.}
\end{cases}
\end{align}
This is well-known, see for example
\cite[(1.5)(3)]{erd1}, although it can also be derived using the
recursions below. 

\subsection{The recursion for $p=2$ } 
When $p=2$ we apply Corollary 4.2 (or Corollary 5.2) in
\cite{parker3}, 
with $N= \Delta(0)$
which is fixed by the Frobenius twist: For $d\geq 1$ and $m\geq 1$ we
have
\begin{equation}\label{eq:recp=2}
\Ext^m(\D(0), \D(2b))\cong \Ext^{m-1}(\D(0), \D(2(b-1)))\oplus
\Ext^m(\D(0), \D(b-1)).
\end{equation}
Note that when $b$ is even, the weights $0$ and $b-1$ are in different
blocks and the second summand is zero.

\subsection{The recursion for $p>2$}
By \cite[Corollary 4.2 or 5.2]{parker3}  we have
\begin{equation}\label{eq:recpnot2}
\Ext^m(\Delta(0), \Delta(2bp)) \cong \Ext^{m-1}(\Delta(0),
\Delta(2bp-2))
\end{equation}
and
$$\Ext^m(\Delta(0), \Delta(2bp-2))\cong \Ext^{m-1}(\Delta(0), \Delta(2(b-1)p))
\oplus \Ext^m(\Delta(0), \Delta(2(b-1))).
$$

\section{The generating function for $p=2$}

In  this section $k$ has characteristic $2$.

\subsection{} We fix an integer $d\geq 0$. Set
$$\ve(2d):= \sum_{m\geq 0} \dim \Ext^m(\Delta(0), \Delta(2d))z^m
$$
which is a polynomial in $z$. 

We translate \eqref{eq:recp=2} into a recursion. Firstly,
$\ve(0)=1$. Next we have:
$$\ve(2d) = \left\{\begin{array}{ll} z\ve(2(d-1))& \mbox{if $d$
      even, $>0$;}\\
 z\ve(2(d-1)) + \ve(d-1) & \mbox{if $d$ odd.}
\end{array}
\right.
$$
We define the generating function, $G(s)$, for the $\ve(2d)$'s by
$$G(s)= \sum_{d\geq 0} s^d\ve(2d).
$$
We now give a functional equation satisfied by $G(s)$.
\begin{lem}\label{lem:functeqp=2}
We have 
$$(1-zs)G(s) = sG(s^2) + 1.
$$ 
\end{lem}
\begin{proof}
By substituting the recursion formula, we get
\begin{equation}\label{eq:genp=2}
G(s) = 1 + \sum_{d\geq 1}zs^d\ve(2(d-1)) + \sum_{d \rm{\ odd}} s^d\ve(d-1)
\end{equation}
We can write this as 
\begin{align*} G(s)=& 1 + sz(\sum_{d\geq 1} s^{d-1}\ve(2(d-1))) + sG(s^2)\cr
=& 1+ szG(s) + sG(s^2).\qedhere
\end{align*}
\end{proof}

\subsection{} We would like to find  the coefficient of $z^m$ in $G(s)$.
This is a power series in $s$, and the coefficient of $s^d$ is $\dim
\Ext^m(\Delta(0), \Delta(2d))$.
We write 
$$
G(s) = \sum_{m\geq 0} z^mg_m(s).
$$

Assume first that $m=0$. By equation \eqref{eq:hom}  we know that
$$
g_0(s) = 1 + s + s^3 + s^7 + \ldots = \sum_{k\geq 0} s^{2^k-1}
$$
Note that this is consistent with setting $z=0$ in equation
\eqref{eq:genp=2}.

\begin{lem}\label{lem:sgnp=2}
We have
the recursion
$$
g_{n}(s) - sg_{n-1}(s) = sg_n(s^2)
$$
Furthermore for $n\geq 1$ we have
$$
sg_n(s) = \sum_{k=0}^{\infty} s^{2^{k}}(s^{2^k}g_{n-1}(s^{2^k})).
$$
\end{lem}
\begin{proof}
Recall that $(1-zs)G(s) = 1+sG(s^2)$. 
The left hand side of this is
$$
\sum_{n\geq 0} z^ng_n(s) - \sum_{n\geq 0} z^{n+1}sg_n(s)
$$
and the right hand side is 
$$
1 + \sum_{n\geq 0} z^nsg_n(s^2).
$$
Equating the coefficients of $z^n$ gives the first part.
For the second part we
multiply with $s$ and then substitute $s^{2^k}$. This gives:
\begin{align*}
sg_n(s) - s^2g_n(s^2) &= s^2g_{n-1}(s)\\
s^2g_n(s^2) - s^4g_n(s^4)&= s^4g_{n-1}(s^2)\\
s^4g_n(s^4) - s^8g_n(s^8)&= s^8g_{n-1}(s^4)\\
&\vdots 
\end{align*}
Adding  these together then gives the statement. 
\end{proof}

\begin{cor}\label{cor:p2dimext}
We can write 
$$
sg_n(s) = \sum_{\beta} 
s^{2^{b_n} + 2^{b_{n-1}} + \cdots + 2^{b_0}}
$$
where the sum is over all 
partitions $\beta = (b_0, b_1, \ldots, b_n)$ of $m$, $m$ taken over
all $\N$.
In particular,
$\dim\Ext^n(k,\Delta(2d))$ is 
equal to 
the number of  partitions $(b_0, b_1, \ldots, b_n)$      
such that $\sum_i 2^{b_i}= d+1$.
\end{cor}
\begin{proof}
 By induction on $n$ we get from Lemma \ref{lem:sgnp=2} that
$$
sg_n(s) = \sum_{0\leq k_i} s^{2^{k_0} + 2^{k_1+k_0} + \cdots +  2^{k_n+k_{n-1} + \cdots k_1 + k_0}}.
$$
Substitute  $b_n=k_0$, $b_{n-1} = k_0+k_1$, and so on, until
$b_0 = \sum_{i=0}^n k_i$.
Then set $\beta = (b_0, b_1, \ldots, b_n)$, this is then a partition; and
all partitions with $n+1$ parts occur. 

The coefficient
of $s^d$ in $g_n(s)$ is equal to 
$\dim \Ext^n(\Delta(0), \Delta(2d)) $.
This is the coefficient of $s^{d+1}$ in $sg_n(s)$ and thus  it
is equal to  
the number of partitions $\beta$ with $n+1$ parts such that
$2^{b_n} + 2^{b_{n-1}} + \ldots + 2^{b_0} = d+1$.
\end{proof}

We also have a formula for $G(s)$ as a sum of rational functions. 
Although we will not apply it here, we include it for completeness.

\begin{lem}
Let $F(s) = \Pi_{k=0}^{\infty} (1-zs^{2^k})$. Then
we have
$$
F(s) sG(s) - F(s^2)s^2G(s^2) = sF(s^2)
$$
and 
$$
G(s) = (1-zs)^{-1}(1 + \frac{s}{(1-zs^2)} + \frac{s^3}{(1-zs^2)(1-zs^4)}
+ \ldots)
$$
\end{lem}
\begin{proof}
By the definition we have $F(s) = (1-zs)F(s^2)$.  Multiplying the
functional equation in Lemma \ref{lem:functeqp=2} with $s$ gives:
$$
(1-zs)sG(s) = s^2G(s^2) + s.
$$
Multiplying  this with $F(s^2)$ gives the first statement:
\begin{align*}
F(s)sG(s)&= (1-zs)F(s^2)sG(s)\\
&= F(s^2)s^2G(s^2) + sF(s^2).
\end{align*}
We now have
\begin{align*} 
F(s) sG(s)-F(s^2)s^2G(s^2)&=sF(s^2)\\
F(s^2)s^2G(s^2)-F(s^4)s^4G(s^4)&=s^2F(s^4)\\
F(s^4)s^4G(s^4)-F(s^8)s^8G(s^8)&=s^4F(s^4)\\
&\vdots \\
\end{align*}
We add these equations and get
$$
F(s)sG(s) = \sum_{k=0}^{\infty} s^{2^k}F(s^{2^{k+1}}).
$$
Hence we have
\begin{align*}
G(s)&= \frac{1}{s}\sum_{0}^{\infty} s^{2^k}\frac{F(s^{2^{k+1}})}{F(s)}
\\
&= \frac{1}{s}\left[\frac{sF(s^2)}{F(s)} + \frac{s^2F(s^4)}{F(s)} +
  \ldots\right]\\
&= \frac{1}{s}\left[\frac{s}{1-zs}+ \frac{s^2}{(1-zs)(1-zs^2)} +
  \ldots\right] \end{align*}
which gives the second part of the statement.
\end{proof}

\section{The generating function for $p>2$}

\subsection{} 
Let $z$ be a variable, 
we  define for a fixed even integer $2d$:
\begin{align*}
\ve(2d):= &\sum_{m\geq 0} \dim\Ext^m(\D(0), \D(2d))z^m 
\end{align*}
which is a polynomial. 
Then $\ve(2d)$ can only be non-zero for $2d$ where $\D(2d)$ is in the
principal block, that
is $2d$ of the form $2kp$ or $2kp-2$.

We translate \eqref{eq:recpnot2} into a recursion for  the $\ve$. 
First, we have the initial  conditions:
$$
\ve(0)=  1  \ \mbox{ and }\ \ve(2p-2) = 1+z.
$$
Next, we have the following recursions
\begin{align}\label{eq:verecpnot2}
\ve(2kp) &= z\ve(2kp-2)  \ (k\geq 1)\\
\label{eq:verecpnot2sec}
\ve(2kp-2) &= z\ve(2(k-1)p)  + \ve(2(k-1)) \ (k\geq 1).
\end{align}

\subsection{}  Let $s$ be a variable, and set
$$G_0(s):= \sum_{k=0}^{\infty} \ve(2kp)s^{kp}, \ \ 
G_1(s):= \sum_{k=1}^{\infty} \ve(2kp-2)s^{kp-1}
$$
and $G(s)=G_0(s) + G_1(s)$, the generating function for the
$\ve$'s. From
\eqref{eq:verecpnot2}
and \eqref{eq:verecpnot2sec}
we get the following recursions
$$G_0(s) = 1 + zsG_1(s)$$
$$G_1(s) = zs^{p-1}G_0(s) + s^{p-1}G(s^p).$$
We write this in vector form. Let $\Phi(s):= \left(\begin{matrix}
G_0(s)\cr G_1(s)\end{matrix}\right)$, then 
\begin{equation}\label{eq:phi}
\Phi(s) = \begin{pmatrix}1\\ 0\end{pmatrix} + \begin{pmatrix} 0 & zs\cr
  zs^{p-1}&0\end{pmatrix}\Phi(s) 
+ s^{p-1}\begin{pmatrix}0\\ 1\end{pmatrix}\begin{pmatrix} 1 &
    1\end{pmatrix}\Phi(s^p)
\end{equation}

\begin{lem}\label{lem:geneqpnot2}
The generating function satisfies the functional equation
$$
(1-z^2s^p)G(s) = (1+zs^{p-1}) + s^{p-1}(1+zs)G(s^p).
$$ 
\end{lem}
\begin{proof}
Set $u:= \begin{pmatrix}1\\ 0\end{pmatrix}$,
  $v:= \begin{pmatrix}0\\ 1\end{pmatrix}$, and set
$w = u+v$. Then define 
$$
A(s):= \left(\begin{matrix} 0 & s \cr s^{p-1} & 0\end{matrix}\right).
$$
 With this, \eqref{eq:phi} becomes 
\begin{equation}\label{eq:phi2}
(1-zA(s))\Phi(s) = u + s^{p-1}vw^T\Phi(s^p).
\end{equation}
Now, $A(s)^2 = s^pI$, and hence
$$
(1+zA(s))(1-zA(s)) = (1-z^2s^p)I.
$$
Using \eqref{eq:phi2}
$$
(1-z^2s^p)\Phi(s) = (1+ zA(s))(u + s^{p-1}vw^T\Phi(s^p)).
$$
Now note that $G(s) = w^T\Phi(s)$, so we get (after premultiplying with $w^T$)
\begin{equation}\label{eq:phi3}
(1-z^2s^p)G(s) = w^T(1+zA(s))u 
+ s^{p-1}w^T(1+zA(s))vG(s^p)
\end{equation}
Now $w^T(1+zA(s))u = 1+zs^{p-1}$ and $w^T(1+zA(s))v= 1+zs$. With this,
\eqref{eq:phi3} becomes the stated formula.
\end{proof}

\begin{rem}
Consider the generating function we had for $p=2$, in
Lemma~\ref{lem:functeqp=2}.
If we multiply both sides in that formula by $(1+zs)$ and set $p=2$
then we get the same as in \ref{lem:geneqpnot2}.
\end{rem}

By  iteration this gives a formula for $G(s)$.

\subsection{} We define functions $h_n(s)$ via
$$
G(s) = \sum_{n\geq 0} h_n(s)z^n.
$$
We are interested in the $h_n(s)$ as  the coefficient of $s^r$ in
$h_n(s)$ is equal to the dimension of 
$\Ext^n(\Delta(0), \Delta(2r))$. 
First we have from  section \ref{sect:somerec} that
$$
h_0(s) = 1 + s^{p-1} + s^{p^2-1} + \ldots = s^{-1}\sum_{k=0}^{\infty}
s^{p^k}.
$$

\subsection{}\label{subsect:h1}
 We now calculate $h_1(s)$. By Lemma~\ref{lem:geneqpnot2}
we have that
$$
h_1(s) = s^{p-1}+ s^ph_0(s^p) + s^{p-1}h_1(s^p)
$$
and we write this in the form
\begin{equation}\label{eq:h1}
sh_1(s) - s^ph_1(s^p) = s^{p}+ s^{p+1}h_0(s^p).
\end{equation}
For each $k\geq 0$ we substitute $s^{p^k}$ into $s$, which gives 
$$
s^{p^k}h_1(s^{p^k}) - (s^{p^k})^ph_1((s^{p^k})^p) = (s^{p^k})^{p}+ 
(s^{p^k})^{p+1}h_0((s^{p^k})^p).
$$
Adding all these equations, most of the 
terms on the left hand side cancel, and we get
\begin{align*} 
sh_1(s)=& \sum_{k\geq 0} s^{p^{k+1}} + \sum_{k\geq 0}
s^{p^k}(s^{p^{k+1}}h_0(s^{p^{k+1}}))\\
& \\ 
=& \sum_{k\geq 0} s^{p^{k+1}} + \sum_{k\geq 0} s^{p^k}(\sum_{r\geq 0}
(s^{p^{k+1}})^{p^r}).\\
\end{align*}
So we have
\begin{equation}\label{eq:sh1}
sh_1(s) =  \sum_{k_1\geq 0} s^{p^{k_1+1}} + 
\sum_{k_i \geq 0} s^{p^{k_1}+p^{k_1+k_0+1}}
\end{equation}
and this agrees with the global description given in \cite[section
  3.5]{erd1}.

\subsection{}\label{subsect:shlabels}
We now note some important conventions.
\begin{enumerate}
\item[(1)]  We always label the exponents so that $k_i\geq 0$.
\item[(2)]  For $sh_n(s)$, there will be in each sum one new parameter, 
which we call $k_n$, and we keep the names of the exponents from
previous sums as they were. This allows us to keep track of where the
parameters came from.
\end{enumerate}

Let $n\geq 2$. Then we get from Lemma~\ref{lem:geneqpnot2} that 
$$
h_n(s) - s^{p}h_{n-2}(s) = s^{p-1}h_n(s^p) + s^ph_{n-1}(s^p)
$$
that is
$$
sh_n(s) - s^ph_n(s^p) = s^{p+1}h_{n-2}(s) + s^{p+1}h_{n-1}(s^p)
$$
With the method as in \ref{subsect:h1} we get 

\begin{equation}\label{eq:hn}
sh_n(s) = \sum_{k_n \geq 0} s^{p^{k_n+1}}[s^{p^{k_n}}h_{n-2}(s^{p^{k_n}})]
+ \sum_{k_n\geq 0} s^{p^{k_n}}[s^{p^{k_n+1}}h_{n-1}(s^{p^{k_n+1}})]
\end{equation}


\subsection{}
Assume $n=2$. We calculate the first sum in \eqref{eq:hn} with $n=2$.
This is 
$$
\sum_{k_2\geq 0} s^{p^{k_2+1}}[\sum_{k_0\geq 0} (s^{p^{k_2}})^{p^{k_0}}]
= \sum_{k_i\geq 0} s^{p^{k_2+1}+p^{k_2+k_0}}.
$$
For the second sum, consider 
$$
s^{p^{k_2+1}}h_1(s^{p^{k_2+1}})=
\sum_{k_1\geq 0}(s^{p^{k_2+1}})^{p^{k_1+1}}
+ \sum_{k_1, k_0\geq 0} (s^{p^{k_2+1}})^{p^{k_1}+p^{k_1+k_0+1}}.
$$
And therefore we have
$$sh_2(s) = \sum_{k_i\geq 0} s^{p^{k_2+1} + p^{k_2+k_0}}
+ \sum_{k_i\geq 0} s^{p^{k_2} + p^{k_2+k_1+2}}
+ \sum_{k_i\geq 0} s^{p^{k_2} + p^{k_2+k_1+1} + p^{k_2+k_1+k_0+2}}.
$$

We now demonstrate how we can use this sum to find possible dimensions
for $\Ext$ groups. Clearly, any two 
terms in the second sum
are distinct. As well, any two terms in the third sum are distinct.
But consider the first sum.
\begin{enumerate}
\item[(1)] Terms with $k_0=0$ and $k_0=2$ coincide. We can write the first sum 
as
$$
s^{p+1} + 2\sum_{k_2\geq 1} s^{p^{k_2+1}+ p^{k_2}}  + \sum_{k_2\geq 0,
  k_0\neq 0, 2} s^{p^{k_2+1}+p^{k_2+k_0}}.
$$
\item[(2)] We compare terms from the first sum and terms from the second sum.
We use $\tilde{k}_i$ for exponents in the second sum.
Suppose
$$p^{\tilde{k}_2}(1+p^{\tilde{k}_1+2}) = p^{k_2}(p+p^{k_0})
$$
Then we must have $k_0\geq 1$ and $\tilde{k}_2 = k_2+1$. Furthermore
$k_0-1=\tilde{k}_1+2$. Conversely if $k_0=\tilde{k}_1+3$ and
$k_2=\tilde{k}_2-1$ then the terms are equal. So the conditions for
equality are precisely
$$k_0\geq 3, \ \ \tilde{k}_2\geq 1$$
Note that since $k_0\geq 3$, this does not interfere with the
equalities in the first sum.
\end{enumerate}

These  results are consistent with the result in 
\cite[section 5]{coxerd}. The parameters
with \break $\dExt^2(\Delta(0), \Delta(t))=2$ are precisely
$t = 2(p^{u+1}+ p^{u+1+a}-1)$ where $a\geq 1$ and $u\geq 0$ and gives
the set denoted by $\Psi^{2,2}(0)$ in \cite[section 5]{coxerd}.
The ones with $a=1$ come from coincidences in the first sum.
The ones with $a\geq 2$ are coincidences from the first and the second
sum. 


\subsection{Towards general $n$}\label{subsect:genn}
  First we simplify notation. 
Take the formula for $sh_n(s)$, it has several sums, labelled
by tuples of $k_i$ and positive integers. 
To keep track over the sums, it suffices to write down the
exponents of $p$ occuring. 
As an example, take formula \eqref{eq:sh1} for $sh_1(s)$.
We label the first sum by
$$(k_1+1)$$
and we label the second sum by
$$(k_1, k_1+k_0+1).$$
We always have $k_i\geq 0$, and the coordinates of the tuples will
be (linear) functions in the $k_i$ which  we denote by $m_i=m_i(k_1,
k_1, \ldots k_j)$. 
So in the above example, we write the second sum
as 
$$
(m_1, m_2) \equiv \sum_{k_i\geq 0} s^{p^{m_1}+p^{m_2}}
$$
where $m_1 = k_1$ and $m_2= k_1 + k_0+1$.

In general we  use therefore as a shorthand notation
$$(m_1, m_2, \ldots, m_t) \equiv \sum_{k_i\geq 0} 
s^{p^{m_1} + p^{m_2} + \ldots + p^{m_t}}
$$
where the $m_j$ are (linear) functions of the $k_i$. Then each $sh_n(s)$ is
identified with a list  of such $(m_1, \ldots, m_t)$.

\begin{eg}\label{eg:lists}
With this notation,
\begin{enumerate}

\item[(1)]  $sh_0(s)$ is identified with $(k_0)$. 

\item[(2)]  $sh_1(s)$ is identified with the list
$$(k_1+1), \ \  (k_1, k_1+k_0+1).
$$

\item[(3)]  $sh_2(s)$ is identified with
$$
(k_2+1, k_2+k_0), \ \ (k_2, k_2+k_1+2), \ \  (k_2, k_2+k_1+1, k_2+k_1+k_0+2).
$$
\end{enumerate}
\end{eg}

Let $n\geq 2$. Then $sh_n(s)$ is obtained from 
$sh_{n-2}(s)$ and $sh_{n-1}(s)$ as described using \eqref{eq:hn}. The
shorthand notation
as above allows us to write down $sh_n(s)$.
\begin{enumerate}
\item[(a)]
 Suppose $(m_1, \ldots, m_t)$ labels a sum in $sh_{n-2}(s)$. This
then leads to one sum in $sh_n(s)$, and it has label
$$(k_n+1, k_n+m_1, \ldots, k_n+m_t).$$
\item[(b)]
Now suppose $(f_1, \ldots, f_u)$ labels a sum in $sh_{n-1}(s)$. This
then leads to one sum in $sh_n(s)$, and it has label
$$(k_n, k_n+f_1+1, \ldots, k_n+f_u+1).
$$
\end{enumerate}

For example, the list in Example~\ref{eg:lists}~(3) is obtained from
the lists in Example~\ref{eg:lists}~(1) and (2) by this process.

This  gives a complete algorithm for writing down $sh_s(n)$ in    
general.

\section{Exponential growth for $p>2$}

%
%
%
%
%
%
%
%
%
%
%
%
%
%
%

For this section, $p$ will be a fixed odd prime.
\subsection{} We first present an example to show how the labels get
more complicated for degree $n=3$.
\begin{eg} Consider $sh_3(s)$. The labels for $sh_3(s)$ which come from
$sh_1(s)$ are
\begin{align*}
(a) \ & (k_3+1, k_3+k_1+1) = k_3(1,1) + k_1(0,1) + (1,1),\cr
(b) \ & (k_3+1, k_3+k_1, k_3+k_1+k_0+1) = k_3(1,1,1) + k_1(0,1,1) +
k_0(0,0,1) + (1,0,1).
\end{align*}

The labels for $sh_3(s)$ which come from $sh_2(s)$ are
\begin{align*} 
(c) \ & (k_3, k_3+k_2+2, k_3+k_2+k_0+1)= k_3(1,1,1) + k_2(0,1,1) + k_0(0,0,1)
+ (0,2,1),\cr
(d) \ & (k_3, k_3+k_2+1, k_3+k_2+k_1+3) =
k_3(1,1,1) + k_2(0,1,1) + k_1(0,0,1) + (0,1,3),\cr
(e) \ & (k_3, k_3+k_2+1, k_3+k_2+k_1+2, k_3+k_2+k_1+k_0+3)\cr
\ \ &= 
k_3(1,1,1,1) + k_2(0,1,1,1) + k_1(0,0,1,1) + k_0(0,0,0,1) + (0,1,2,3).
\end{align*}
\end{eg}

\subsection{}\label{subsect:labels} We can write down $sh_4(s)$ by
first taking $sh_2(s)$, and replace
a label of length $r$ by the label of length $r+1$, obtained by adding
$$k_4(1, 1, \ldots, 1) + (1, 0, 0, \ldots 0);$$ 
and then taking $sh_3(s)$, and
replace a label of length $r$ by the label of length $r+1$ obtained by adding
$$k_4(1, 1, \ldots, 1) + (0, 1, \ldots, 1).
$$
Similarly we can write down a list of labels for $sh_n(s)$ from those of
$sh_{n-2}(s)$ and of $sh_{n-1}(s)$. 
Note that each label of length $r+1$ say  for $sh_n(s)$ is of the form
$$k_n(1, 1, \ldots, 1) + k_{i_1}(0,1, 1, \ldots, 1) + \ldots
+ k_{i_r}(0, 0, \ldots, 0, 1) + (a_0, \ldots , a_r),
$$
where $n > i_1 > i_2 > \ldots >i_r \geq 0$ and where the $a_i$ are
non-negative integers $\leq n$.

\subsection{} Let 
$t_n$ be the number of labels for $sh_n(s)$, then we have the recursion
$$t_n = t_{n-1} + t_{n-2}
$$
and $t_n= F_{n+1}$, the $(n+1)$-th Fibonacci number. This shows that the
number of labels grows exponentially, as the Fibonacci sequence is
known to have exponential growth.

\subsection{} Now we find the length of the labels and the number of 
labels for some specific length.
This will show that for each $n$ there is a number $v_n$
such that the number of labels of length $v_n$ grows polynomially.

For each $n$, let $\cL(n)$ be the set of lengths for the labels. 
So we have
$$\cL(0)= \{1\}, \ \ \cL(1)= \{ 1, 2\}, \ \ \cL(2) = \{ 2, 3 \}
$$

\begin{lem} 
(a) \ $\cL(2t-1) = \{ t, t+1, \ldots, 2t\}$\\
(b) \ $\cL(2t) = \{ t+1, t+2, \ldots, 2t+1 \}$ 
\end{lem}

\begin{proof}
 Induction on $n$. We always have
\begin{equation} \label{eq:star}
\cL(n) = \{ x+1: x \in \cL(n-1) \cup \cL(n-2)\}
\end{equation}
 
Let $t=1$, we have $\cL(1) = \{ 1, 2\}$.
As well $\cL(2) = \{ 2, 3\}$, as stated.

Assume true for $t$, then by the general reduction
$$\cL(2(t+1)-1) = 
\{ t+1, t+2, \ldots, 2t+1\} \cup \{ t+2, t+3, \ldots, 2t+2 \}
$$
with smallest $t+1$ and largest $2(t+1)$ as stated. 
Next we have
$$\cL(2(t+1)) = \{ t+2, t+3, \ldots, 2t+2\} \cup
\{ t+2, t+3, \ldots, 2t+3\} 
$$
with smallest $(t+1)+1$ and largest $2(t+1)+1$, as required.
\end{proof}

The largest length of labels in $\cL(n)$ is 
therefore $n+1$.

\subsection{ } 
 We fix $n$, and count labels of a given length. 

\begin{lem} (a) There is one label of length $n+1$.

(b) The number of labels of length $n-i$ is ${n-i\choose i+1}$ for
  $i=0, 1, \ldots$. 
\end{lem}

\begin{proof}
(a) \ We see this from \eqref{eq:star} above, by induction.

(b) We use induction on $n$. The cases $n=1, 2, 3$ are clear from the examples.
For the inductive step, by \eqref{eq:star}, 
the number of
labels of length $n-i$ in degree $n$ 
is equal to

$\#$ labels of length $n-(i+1)$ in degree $n-2$ \ $+$ \ $\#$ labels of length
$n-(i+1)$ in degree $n-1$. By the inductive hypothesis, this is
$${n-(i+1)\choose i} + {n-(i+1)\choose (i+1)} = {n-i\choose i+1}.
\qedhere
$$
\end{proof}


\begin{lem}\label{lem:expgrowth}
 Fix $n$ and $p$ an odd prime. Let $v=v_k$ be the number of labels of
length $k$. Then there is a weight 
$m= p^{b_1} + p^{b_2} + \ldots + p^{b_k}$ such that
$$\dim \Ext^n(\D(0), \D(2(m-1))\geq v$$ 
\end{lem}

\begin{proof}
 Each label is of the form 
$$\sum_j k_{i_j}(0, \ldots, 0, 1, 1, \ldots 1) + (a_1, a_2, \ldots, a_k)
$$
(as described in \ref{subsect:labels} above). We can find 
$\beta:=(b_1, b_2, \ldots, b_k) $ such that
$0 \leq b_1-a_1 \leq b_2-a_2 \leq \ldots b_k-a_k$ for each 
$(a_1, \ldots, a_k)$ occuring in the list of labels. For such
$\beta$, we can in each label find a solution  for the $k's$.

Let $m=\sum p^{b_i}$. Then the coefficient of $s^{m-1}$ in $sh_n(s)$ is
at least $v$. Hence the claim follows, with this $m$.
\end{proof}

Since the number of labels of the same length $k$ grows polynomially
this shows that we have at least polynomial growth in $n$ of
$\dim \Ext^n(k, \Delta(a))$. The degree of these polynomials is
unbounded however.


  We present a small example for  a 3-dimensional $\Ext^3$ space:
\begin{eg}
Consider prime $p=3$ and weight $76$. We have $76=2\cdot 38$. Using
the function $sh_3(s)$, the dimension of
$\Ext^3(\Delta(0), \Delta(76))$ is equal to the number of
times $s^{38+1}$ occurs. Now, 
      $$39 = 27+9+3 = 3^3 + 3^2 + 3^1
$$
 We want to find all solutions for the $k_i$ to get $\{ 3, 2, 1\}$.

The label in (a) has only two terms, so this does not contribute.

Consider the label in (b). The last entry is at least as big as the
other two, so
if there is a solution we must have
$$k_3+k_1+k_0+1=3$$
Then there are two possibilities, 
\begin{enumerate}
\item[(i)]
 $k_3+1=2$ and $k_3+k_1=1$, or
\item[(ii)]
$k_3+1=1$ and $k_3+k_1=2$. 
\end{enumerate}
In the first case, $(k_3, k_2, k_0) = (1, 0, 1)$. In the second case
$(k_3, k_1, k_0) = (0, 2, 0)$. 

Consider the label in (c). The first entry is at most as big as the
others, so we must
have $k_3=1$, and then there is a unique solution,
$(k_3, k_2, k_0)= (1, 0, 0)$. 

For label (d) there is no solution, and label (e) has too many terms so there
is also no solution. 

Thus, we have $3$ possible solutions to get $\{3,2,1\}$ and thus
$\dim\Ext^3(k,\Delta(76)) =3$.
\end{eg}

\subsection{}
We now show that the maximum of the number of the
labels
i.e. $\max_{a \in \N}\{ {n-a \choose a+1}\}$ grows
exponentially. 

Firstly we have:
\begin{lem}
$$
 {n-a \choose a+1}
\ge 
 {n-a+1 \choose a}
\mbox{ if } a < \left( \frac{1}{2} - \frac{\sqrt{5}}{10}\right)n - 1
$$
and
$$
 {n-a \choose a+1}
\le 
 {n-a+1 \choose a}
\mbox{ if } \frac{n}{2} \ge  a > \left( \frac{1}{2} -
\frac{\sqrt{5}}{10}\right)n.
$$
\end{lem}
\begin{proof}
For $a \le \frac{n+1}{2}$
$$
 {n-a \choose a+1} \div
 {n-a+1 \choose a}
=\frac{(n-2a +1)(n-2a)}{(a+1)(n-a+1)}.
$$
Now
$$
(n-2a +1)(n-2a) - (a+1)(n-a+1) =
n^5 - 5an + 5a^2 -2a -1
=
5a^2 -a(5n +2) + n^2 - 1.
$$
Since $n$ is fixed, we consider this as a parabolic in $a$. This has two
roots and will be positive outside the roots and negative inside. 
The roots are not so easy to work out so we show that
 if $a = n \left( \frac{1}{2} - \frac{\sqrt{5}}{10}\right)n$ 
then
$$
5a^2 -a(5n +2) + n^2 - 1 
=
-n \left( 1 - \frac{\sqrt{5}}{5}\right) -1 < 0
$$
for $n >0$ as  $1 > \frac{\sqrt{5}}{5}$.
If $a = n \left( \frac{1}{2} - \frac{\sqrt{5}}{10}\right)n -1$ 
then
$$
5a^2 -a(5n +2) + n^2 - 1 
=
n \left( 3 - \frac{2\sqrt{5}}{5}\right) +6 > 0
$$
for $n >0$ as  $3 > \frac{2\sqrt{5}}{5}$.
So there is a root between these two values. 

If $a= \frac{n}{2}$ then
$$
5a^2 -a(5n +2) + n^2 - 1 
=-\frac{n^2}{4} -n-1 <0
$$
for $n >0$ so the next root is past this value.

Thus
$$
 {n-a \choose a+1} \div
 {n-a+1 \choose a}
\begin{cases}
> 1 & \mbox{ if $a < \left( \frac{1}{2} - \frac{\sqrt{5}}{10}\right)n - 1$}\\
< 1 & \mbox{ if $ \frac{n}{2} \ge a >  \left( \frac{1}{2} -
  \frac{\sqrt{5}}{10}\right)n $}
 \qedhere
\end{cases}
$$
\end{proof}

Thus the maximum value of $ {n -1 \choose a+1} $ for fixed $n$ and
varying $a$ is
  achieved around 
$
a = \left\lfloor  \left( \frac{1}{2} - \frac{\sqrt{5}}{10}\right)n
\right\rfloor.
$

We now use Stirling's formula to estimate the growth rate of this in
$n$. Note: since we are estimating growth rates and binomial
coefficients can be defined for non-integers we are going to ignore
the fact that $a$ is not an integer in the next Lemma.
\begin{lem}\label{lem:growthbinom}
Let $a=   \left( \frac{1}{2} - \frac{\sqrt{5}}{10}\right)n -1$
then the growth rate of 
${n-a \choose a+1}$ is exponential in $n$.
\end{lem}
\begin{proof}
Stirling's formula says
$n! \sim \sqrt{2 \pi n} \left(\frac{n}{e}\right)^n$. (The $\sim$
denotes asymptotically equal.)
Set $A :=  \frac{1}{2} - \frac{\sqrt{5}}{10}$
and
$\bar{A} :=  \frac{1}{2} + \frac{\sqrt{5}}{10}$ so
$A + \bar{A} =1$.
We have
\begin{align*}
{n-a \choose a+1}
&=
\frac{(n- A n +1)!}{(An)!(n-2An+1)!}\\
&\sim 
\frac{1}{\sqrt{2\pi}}
\sqrt{\frac{(\bar{A} n +1) }{(An)(\frac{\sqrt{5}}{5} n +1)}}
\frac{(\bar{A} n + 1)^{(\bar{A} n +1)}}
{(An)^{(An)} (\frac{\sqrt{5}}{5} n +1)^{(\frac{\sqrt{5}}{5} n +1)}}\\
\intertext{set $C := \frac{\sqrt{5}}{5} = \bar{A} - A$:}
&= 
\frac{1}{\sqrt{2\pi}}
\frac{(\bar{A} n + 1)^{(\bar{A} n +\frac{3}{2})}}
{(An)^{(An + \frac{1}{2})} (C n +1)^{(C n +\frac{3}{2})}} =: f(n)
\end{align*}
So what growth rate does this function $f$ have?
For a crude approximation to first order for very large $n$, $f$ is
after setting 
$D :=
\frac{\bar{A}^{\frac{3}{2}}}{C^{\frac{3}{2}}A^{\frac{1}{2}}}$ 
and 
$E := \frac{\bar{A}^{\bar{A}}}{A^A C^C}$:
$$
f(n) \approx
\frac{1}{\sqrt{2\pi}}
\frac{(\bar{A} n +\frac{3}{2})^{(\bar{A} n +\frac{1}{2})}}
{(An)^{(An)} (C n)^{(C n + \frac{3}{2})}} 
= 
\frac{D}{\sqrt{2\pi}}
E^n
 n^{((\bar{A} - A -C)n - \frac{1}{2})}
=
\frac{D}{\sqrt{2\pi}}
E^n
 n^{ - \frac{1}{2}}.
$$
Which, as $E >1$,  means this has exponential growth in $n$.
\end{proof}

\begin{cor}
We may find weights $a_n$ such that the sequence
$\dim\Ext^n(k, \Delta(a_n))$ grows exponentially.
\end{cor}

\section{Exponential growth for $p=2$}

We analyse the functions $sg_n(s)$.

\subsection{}
In Corollary \ref{cor:p2dimext} we found an explicit formula for
$\dim\Ext^i(\Delta(0), \Delta(2d))$. We now compare this result to
that found in Stewart.

Let $n=2m$ where $m>2$. 
In \cite[Theorem 2]{stewart} it is stated that 
$$\dim \Ext^{2m}(\Delta(0), L(2^{2m})) \geq 2^{m-1}.
$$
In fact he proves that
$\dim \Ext^{m}(\Delta(0), L(2^{m}))
$ 
is equal to  
the number of partitions of $1$ into $m$ powers of
$\frac{1}{2}$, \cite[section 2]{stewart}. Let's call this number
$\Pi_{m-1}$.
Now we have that $\dim \Ext^m(\Delta(0), \Delta(2^m))$ is the 
number of partitions 
$\beta=(b_0,b_1,b_2,\ldots,b_m)$ with $m+1$ parts 
such that
$2^{b_m} + 2^{b_{m-1}} + \ldots + 2^{b_0} = 2^m+1$.
Note that both sides of this equation are $1$ modulo $2$ for $m \ge 1$. This
means in particular that at least one of the $b_i$ on the LHS is $1$.
Thus for this particular weight, the dimension is equal to the number
of partitions $\beta'=(b_0, b_1,\ldots,b_{m-1})$ with 
$2^{b_{m-1}} + 2^{b_{m-2}} + \ldots + 2^{b_0} = 2^m$. Also note that
$2^b_i \le 2^m$, i.e. that $b_i \le m$.
Clearly,
$$
2^{b_{m-1}} + 2^{b_{m-2}} + \ldots + 2^{b_0} = 2^m
\Leftrightarrow
2^{-(m-b_{m-1})} + 2^{-(m-b_{m-2})} + \ldots + 2^{-(m-b_0)} = 1 
$$
I.e. the number of such partitions is exactly $\Pi_{m-1}$.
That these numbers should be equal is not so suprising at least in one
case.
Namely
from the structure of $\Delta(2d)$ for $d=2^{m}$, $m\ge 1$ we can deduce that
$\Ext^{2m}(\Delta(0), L(2d))\cong \Ext^{2m}(\Delta(0), \Delta(2d))$. 
That is we use that the radical of $\Delta(2d)$ is isomorphic to a
dual Weyl module $\nabla(2(d-1))$, and 
then note that $\Ext^i(\Delta(0), \nabla(2(d-1)))=0$.
To see that the structure of $\Delta(2d)$ is as claimed note that we
have
a (well-known) short exact sequence
$$
0 \to L(2^m-1)\frob \to \Delta(2^{m+1}) \to \Delta(2^m)\frob \to 0
$$
where $\frob$ is the twist by the Frobenius morphism.
The result then follows by induction and using that
$\nabla(2^m-2)$ is the only module with simple socle $L(2^m-1)\frob$
and the required character.


We can show that the dimensions of $\Ext$ spaces with similar parameters
are also large. We have the following:

\begin{prop}\label{prop:dim} Let $t\geq 2$. Then the dimension of
$\Ext^{2t-2}(\Delta(0), \Delta(2\cdot 2^t))$ is  at least $2^{t-2}$. 
\end{prop}

This is similar to \cite[Theorem 2]{stewart} but with a different
weight for the second Weyl module. We now give a proof in our setup.

Recall that the dimension of this $\Ext$ space is the number of
expansions of $2^t+1$ of length $2t-2+1$, i.e. (assuming $t\geq 2$)
equal to  the number of expansions
of $2^t$ of length $2t-2$. 

We define the length of the expansion
$2^t= \sum_i 2^{b_i}$ to be the number of $b_i$ occuring (including zeros).

\begin{lem}
Fix an integer $m\geq 4$. Let $\cM_r$ be the
set of all expansions of $2^m$ of length $r$. Then 
$$|\cM_{2m-2}|\leq |\cM_{2m-1}|$$ 
\end{lem}


\begin{proof}
 We will define  a map $\Psi: \cM_{2m-2} \to \cM_{2m-1}$
and show that it is 1-1.

We claim that $b_1$ (the first term of the partition) must be at least
$2$. 
If not then all $b_i$ are at most $1$ and it follows that
$$2^m = \sum 2^{b_i} \leq (2m-2)\cdot 2$$
and therefore $m\leq 3$, which contradicts the assumption.
Hence define
$\Psi(\sum_i 2^{b_i})$ to be the 
expansion obtained from $\sum 2^{b_i}$ when  replacing $2^{b_1}$ by
$2^{b_1-1} + 2^{b_1-1}$. We claim that the map is 1-1. 

To do so, we write the  partition $\beta = (b_i)$ as  
$(c_1^{r_1}, c_2^{r_2}, \ldots, c_w^{r_w})$ with $c_1 > c_2 > \ldots>  c_w \geq 0$ 
(so that $r_i\geq 1$ and $\sum r_i = 2m-2$). 

With this, 
$$\Psi(c_1^{r_1}, c_2^{r_2}, \ldots, c_w^{r_w})
= \left\{\begin{array} {ll} (c_1^{r_1-1}, (c_1-1)^2, c_2^{r_2}, \ldots ) 
& c_1-1 > c_2 \cr
(c_1^{r_1-1}, c_2^{r_2+2}, \ldots) & c_1-1=c_2
\end{array}
\right.
$$
Then it is easy to see that the map $\Psi$ is 1-1 (note that $r_2+2>2$). 
\end{proof}

\begin{proof}[Proof of Proposition \ref{prop:dim}]
 We show by induction on $t$ 
that
$$
\left.
\begin{aligned}
&\mbox{(a)\ The number of expansions of $2^t$ of length
  $2t-2$ is $\geq 2^{t-2}$, and\ }\\
&\mbox{(b)\ The number of expansions of $2^t$ of length $2t-1$
    is $\geq 2^{t-2}$.}
\end{aligned}
\right\}
{\bf (P_t)}  
$$

We start with $t=2$. Then we have the expansions
$$4 = 2^1 + 2^1, \ \ 4 = 2^1 + 2^0 + 2^0
$$
and $2^{t-2}=1$.

For the inductive step, assume ${\bf (P_m)}$ holds for $2\leq m <
t$. Now consider
expansions of $2^t$ of length $2t-2$, ie we want to verify part (a)
of ${\bf (P_t)}$ (then (b) will follow, by the Lemma).

\begin{enumerate}
\item[(1)] We have $2^t= 2^{t-1} + 2^{t-1}$. For each expansion 
of $2^{t-1}$ of length $2t-3$ we have one of $2^t$ of length $2t-2$, by taking
$b_1=t-1$. Different such expansions of $2^{t-1}$ give different
expansions of $2^t$. Note that $2t-3= 2(t-1)-1$. Hence by induction 
(using (b) for $(P_{t-1})$)  the number of such expansions is $\geq 2^{t-3}$. 
\item[(2)]
  Next, we have $2^t= 2^{t-2} + 2^{t-2} + 2^{t-2} + 2^{t-2}$. For
each expansion of $2^{t-2}$ of length $2t-5$ we get one 
expansion of $2^t$, by taking $b_1=b_2=b_3 = 2^{t-2}$, and again
different  expansions of $2^{t-2}$ give different expansions of $2^t$. 
Note that $2(t-2)-1 = 2t-5$ and hence by induction the number of such
expansions is $\geq 2^{t-4}$. 
\end{enumerate}

We continue this way. In step $s$, we write
$$2^t = [2^{t-2}+ 2^{t-2}] + [2^{t-3}+2^{t-3}] + [2^{t-4} + 2^{t-4}] + \ldots
+ [2^{t-s}+2^{t-s} + 2^{t-s} + 2^{t-s}]
$$
This is an expansion of length  $2s$. We replace the last term $2^{t-s}$
by an expansion into $2(t-s)-1$ terms and this gives an expansion of
$2^t$ of length $2s-1+2(t-s)-1 = 2t-2$. By induction the number of distinct
such expansions is $\geq 2^{t-s-2}$. 

These expansions are different from earlier expansions, as one sees by comparing
the largest exponents.

At step $s=t-2$ we replace $2^2$ by expansions of length $3$ and there is
just $1$ (and $1= 2^{0}$). 
We do one more step, and replace $2$ by expansions of length $1$, and there
is just one such expansion.

In total we have produced a list of distinct expansions
of $2^t$ of length $2t-2$, and the number of these expansions
in our list is at least
$$2^{t-3} + 2^{t-4} + \ldots + 2^2 + 2 + 1 + 1
= 2^{t-2}.
$$
This proves the Proposition.
\end{proof}

\begin{rem}\normalfont Using section 4 in \cite{erd1} one gets dimensions
of Ext groups between some Weyl modules for $G$, an arbitary algebraic
group.
Namely, if $\lambda, \mu$ are dominant 
weights such that $\mu - \lambda$ differ by $2d\alpha$ where
$\alpha$ is a simple root, and $d\geq 0$ then
$$\Ext^m_{G}(\D(\lambda), \D(\mu)) \cong 
\Ext^m_{SL_2}(\D(\langle \lambda, \alphac\rangle ), \D(\langle \mu,
\alphac\rangle))
$$
This then shows that
$$
\max \{ \dim \Ext^i_{G}(\D(\lambda),
\D(\mu))\mid \lambda, \mu  \in X^+, i  \in \N \}$$
 has (at least) exponential growth for all $p$.
\end{rem}

\section{Bounding $\Ext$}\label{sect:boundext}

In this section we show that when we fix $n$ and $p$, the dimensions 
of $\Ext^n(k,
\Delta(2d))$ have a bound independent of $d$.

\subsection{}
We fix a prime $p$ and 
consider the dimensions of ${\rm Ext}^n(k, \Delta(2d))$ for a fixed $n$ 
as $d$ varies. 
We have two cases.
\begin{enumerate}
\item[$p=2$:] The dimension is equal to the number of partitions
 $(b_0, b_1, \ldots, b_n)$ with $b_0\geq b_1\geq \ldots \geq b_n\geq 0$ such
that 
$\sum 2^{b_i} = d+1$. 

\item[$p>2$:] 
The dimension is bounded above by the number
of compositions $(m_1, \ldots, m_l)$ whose length $l$ satisfies
\begin{enumerate}
\item[(i)] if $n=2t-1$, then $t\leq l\leq 2t(=n+1)$, 
\item[(ii)] if $n=2t$, then $t+1\leq l\leq 2t+1 (=n+1)$
\end{enumerate}
and such that $\sum p^{m_i} = d+1$. 
\end{enumerate}

That is we need to to understand the following set:
Let $p\geq 2$ be a prime.
Fix $n$, and take some number $l\leq n+1$,
then for each $d\in \N$, let
$$\cN_d:= \{  (m_1, \ldots, m_l): 0\leq m_j, \sum_j p^{m_j} = d+1\}
$$
We now show that  this set is bounded.

\begin{prop} The size  $|\cN_d|$ is bounded in terms of $n$, independent
of $d$. 
\end{prop}

\begin{proof}
Firstly we reformulate the problem. We can write
$$\sum_{j=1}^l p^{m_j}
\ = \ \sum_{i=1}^r a_ip^{s_i} 
$$
where  $0 < a_i$ is the number of times $m_j =s_i$, and 
$0\leq s_1<s_2< \ldots < s_r$. Then  $\sum a_i = l$. 

Fix 
$(a_i)$ with $0< a_i$ and $\sum a_i=l$, and define
$$\mathcal{N}_d(a_i):= \{ (s_1, s_2,  \ldots, s_r): \sum_{i=1}^r a_ip^{s_i}=d+1,
0\leq s_1<s_s< \ldots < s_r\}
$$
Then we have $\cN_d = \bigcup \cN_d(a_i)$, the union over all such $(a_i)$. 
For a fixed $n$, the number of ways writing $\sum a_i = l$ where $l\leq n+1$ with
$0< a_i$ is bounded in terms of $n$.  So it suffices to show the following.

We now claim that 
for
 any  $(a_i)$ the 
size of $\cN_d(a_i)$   is bounded in terms of $n$, independent of $d$.

We show this by  induction on $n$; the case  $n=1$ is clear.
For the inductive step, we consider first the case when 
$p$ does not divide $a_i$ for any $i$. 

We first show that if  $p$ does not divide $a_i$ for all $i$
 then $|\mathcal{N}_d(a_i)|\leq 1$.

Suppose the set is not empty. Assume $ (t_i)$ and $(j_i)$ are
in $\mathcal{N}_d(a_i)$, so that
$$d+1 = \sum_i a_ip^{t_i}  = \sum_i a_ip^{j_i}
$$
and, say, $t_1 \leq j_1$. The p-part on the LHS is $p^{t_1}$ and
the p-part on the RHS is $p^{j_1}$. So $t_1=j_1$. 
Subtract the lowest term, and repeat the argument. 
This shows $t_i=j_i$ for all $i$.

Now let's assume that
at least one  $a_i$ is divisible by $p$.  Write
$a_i = b_ip^{u_i}$ with $b_i$ not divisible by $p$,  and then
$u_i>0$ for some $i$ and therefore $\sum b_i < \sum a_i$, and as well
$b_i>0$. So 
we can then  use the inductive hypothesis.
To do so, we must arrange the sum by increasing p-powers:

If $(j_1, \ldots, j_r)$ is in $\mathcal{N}_d(a_i)$ then 
there is a permutation $\pi$ of $r$ such that
$$u_{\pi(s)}+j_{\pi(s)} \leq u_{\pi(s+1)} + j_{\pi(s+1)} \ \ (1\leq s\leq r)
\leqno{(\dagger)}$$
so by 
permuting with $\pi$ we can make the powers of $p$ weakly increasing.
Then we can combine summands for the same power of $p$, this gives new
coefficients of the form $f_{\tau}$ where each $f_{\tau}$ is a sum of
some of
the $b_i$, and $\sum f_{\tau} = \sum b_i$. 

Explicitly, 
write $c_i = b_{\pi(i)}$, and write 
$$(\sum a_ip^{j_i} = ) \ \sum_i c_i p^{u_{\pi(i)}+ j_{\pi(i)}} =
\sum_{\tau} f_{\tau} p^{s_{\tau}}
\leqno{(*)}
$$
where $f_{\tau}$ is the sum of all $c_i$ such that
$u_{\pi(i)} + j_{\pi(i)} = s_{\tau}$ and where the $s_{\tau}$ are strictly 
increasing. Then  $(s_{\tau})$ belongs to 
the set $\cN_d(f_{\tau})$. Now note that $\sum f_{\tau} = \sum_i c_i < l$. By 
induction, the size of $\cN_d(f_{\tau})$ is bounded in terms of $l$,
independent of $d$.
Hence for each permutation $\pi$ the set of expansions with 
$(\dagger)$ is bounded in terms of $n$. The number of permutations needed
is at most $r!$ and $r\leq n$, hence the size of $\cN_d(a_i)$ is bounded
in terms of $n$. 
\end{proof}

\begin{cor}
Fix both $i$ and $p$. Then there is an upper bound for the
 dimension of $\Ext^i(k,\Delta(a))$ for any $a \in \N$.
\end{cor}

\subsection{}
These results lead to several intriguing questions.
Firstly, 
it would be interesting to determine the general behaviour of the
$\Ext$ groups and not just cohomology. In the $\SL_2$ case, we would
not expect the behaviour to be much different.
This is because $\Ext$ groups are essentially the same as cohomology.
Recall from 
\cite[corollary 5.2]{parker3} for $a,b, i \in \N$ and $0 \le i\le p-2$:
$$
\Ext^m(\Delta(pa+i), \Delta(pb +i )) \cong 
\begin{cases}
\Ext^{m}(\Delta(0), \Delta(p(b-a)))  & \mbox{if $b-a$ is even} \\
0 & \mbox{otherwise}.
\end{cases}
$$
and
\begin{multline*}
\Ext^m(\Delta(pa+i), \Delta(pb+p-i-2) )\\ \cong 
\begin{cases}
\Ext^{m-1}(\Delta(0), \Delta(p(b-a-1)) ) 
\oplus \Ext^m(\Delta(a), \Delta(b-1)) & \mbox{if $b-a$ is odd} \\
0 & \mbox{otherwise}.
\end{cases}
\end{multline*}
Thus for weights $pa+i$, $pb+i$ with even difference, their $\Ext$
group is a cohomology group. 
For weights $pa+i$, $pb+p-i-2$ with odd difference, their $\Ext$ 
is almost a cohomology group but with an ``error term''
$\Ext^m(\Delta(a), \Delta(b-1))$. But this group is either zero, or
can be written in terms of lower cohomology groups inductively. While
we have not explored this further, this should not affect the growth
rate significantly compared to the growth in $\Ext^{m-1}(k,
\Delta(p(b-a-1))$ so we would not expect significantly different
results for more general $\Ext$ groups between Weyl modules.

This phenomenon of being essentially determinable from cohomology
groups may not continue for more general algebraic groups. So
it's still possible for greater than exponential growth over all the $\Ext$
groups between Weyl modules for larger algebraic groups.

We may also considering bounding $\Ext$ groups between Weyl modules
for more general algebraic groups. I.e. is there a bound $c \in \N$
for which 
$\dim \Ext_G^i(\Delta(\lambda), \Delta(\mu)) \le c$ for fixed $i$ and
any $\lambda$, $\mu$ a dominant weight?
We have no feeling for this, but note that the bounds that work for
$\SL_2$ are known not to work for $\SL_3$ and hence for larger
algebarics groups. 
Although all
the $\Hom$ spaces are $1$-dimensional for $\SL_3$,
there are
$2$-dimensional $\Ext^1$ groups between Weyl modules. 
Also, it is known that there are larger than $1$-dimensional $\Hom$ spaces
between Weyl modules for larger algebraic groups. Thus the question of
boundedness is very much wide open.


\begin{thebibliography}{10}

\bibitem{coxerd}
A.~G. Cox and K.~Erdmann, \emph{On {${\mathrm{Ext}}^2$} between {W}eyl modules
  for quantum {${\mathrm{GL}}_n$}}, Math. Proc. Cambridge Philos. Soc.
  \textbf{128} (2000), 441--463.

\bibitem{coxpar}
A.~G. Cox and A.~E. Parker, \emph{Homomorphisms between {W}eyl modules for
  $\mathrm{SL}_3(k)$}, Trans. Amer. Math. Soc. \textbf{358} (2006), no.~9,
  4159--4207.

\bibitem{erd1}
K.~Erdmann, \emph{{${\mathrm{Ext}}^1$} for {W}eyl modules for
  {${\mathrm{SL}}_2({K})$}}, Math. Z. \textbf{218} (1995), 447--459.

\bibitem{jantz2}
J.~C. Jantzen, \emph{{R}epresentations of {A}lgebraic {G}roups}, Mathematical
  surveys and monographs, vol. 107, AMS, 2003, second edition.

\bibitem{parker1}
A.~E. Parker, \emph{The global dimension of {S}chur algebras for
  {$\mathrm{GL}_2$} and {$\mathrm{GL}_3$}}, J. Algebra \textbf{241} (2001),
  340--378.

\bibitem{parker2}
\bysame, \emph{On the good filtration dimension of {W}eyl modules for a linear
  algebraic group}, J. reine angew. Math. \textbf{562} (2003), 5--21.

\bibitem{parker3}
\bysame, \emph{Higher extensions between modules for {$\mathrm{SL}_2(k)$}},
  Adv. Math \textbf{209} (2007), 381--405.

\bibitem{PSboundext}
B.~J. Parshall and L.~L. Scott, \emph{Bounding {E}xt for modules for algebraic
  groups, finite groups and quantum groups}, Adv. Math. \textbf{226} (2011),
  no.~3, 2065--2088.

\bibitem{PSisrael}
\bysame, \emph{Cohomological growth rates and {K}azhdan--{L}usztig
  polynomials}, Israel J. Math. \textbf{191} (2012), 85--110.

\bibitem{stewart}
D.~I. Stewart, \emph{Unbounding{ $\mathrm{Ext}$}}, J. Algebra \textbf{165}
  (2012), 1--11.

\end{thebibliography}

\providecommand{\bysame}{\leavevmode\hbox to3em{\hrulefill}\thinspace}
\providecommand{\MR}{\relax\ifhmode\unskip\space\fi MR }
\providecommand{\MRhref}[2]{%
  \href{http://www.ams.org/mathscinet-getitem?mr=#1}{#2}
}
\providecommand{\href}[2]{#2}

\end{document}